\documentclass[11pt]{article}
\usepackage{amsfonts,epsf,amsmath,amssymb}
\usepackage{epsfig}
\usepackage{latexsym}
\usepackage{pgf,tikz}
\usepackage{bm}
\usepackage{lineno}

\newtheorem{theorem}{\bf Theorem}[section]
\newtheorem{corollary}[theorem]{\bf Corollary}
\newtheorem{lemma}[theorem]{\bf Lemma}

\newtheorem{conjecture}[theorem]{\bf Conjecture}

\newcommand{\qed}{\hfill $\square$ \bigskip}

\begin{document}
\date{}
\title{The existence of perfect codes in a family of generalized Fibonacci cubes}

\author{
Michel Mollard\footnote{Institut Fourier, CNRS, Universit\'e Grenoble Alpes, France email: michel.mollard@univ-grenoble-alpes.fr}
}

\maketitle

\begin{abstract}
\noindent The {\em Fibonacci cube} of dimension $n$, denoted as $\Gamma_n$,  is the subgraph of the $n$-cube $Q_n$ induced by vertices with no consecutive 1's. In an article of 2016 Ashrafi and his co-authors proved the non-existence of perfect codes in  $\Gamma_n$ for $n\geq 4$. As an open problem the authors suggest to consider the existence of perfect codes in generalization of Fibonacci cubes. The most direct generalization is the family $\Gamma_n(1^s)$ of subgraphs induced by strings without $1^s$ as a substring where $s\geq 2$ is a given integer. We prove the existence of a perfect code in $\Gamma_n(1^s)$ for $n=2^p-1$ and $s \geq 3.2^{p-2}$ for any integer $p\geq 2$.
\end{abstract}


\noindent
{\bf AMS Subj. Class. }: 94B5,0C69

\section{Introduction and notations}
Let $G$ be a connected graph. The \emph{open neighbourhood} of a vertex $u$ is $N(u)$ the set of vertices adjacent to $u$. The \emph{closed neighbourhood} of $u$ is $N[u]=N(u)\cup\{u\}$. The \emph{distance} between two vertices noted $d_G(x,y)$, or $d(x,y)$ when the graph is unambiguous, is the length of the shortest path between $x$ and $y$. We have thus $N[u]=\{v \in V(G);d(u,v)\leq1\}$.

A \emph{dominating set} $D$ of $G$ is a set of vertices such that every vertex of $G$ belongs to the closed neighbourhood of at least one vertex of $D$. 
In \cite{Biggs}, Biggs initiated the study of perfect codes in graphs a generalization of classical 1-error perfect correcting codes. A \emph{code} $C$ in $G$ is a set of vertices $C$ such that for all pair of distinct vertices $c,c'$ of $C$ we have $N[c]\cap N[c']=\emptyset$ or equivalently such that $d_G(c,c')\geq3$.

A \emph{perfect code} of a graph $G$ is both a dominating set and a code. It is thus a set of vertices $C$  such that every vertex of $G$ belongs to the closed neighbourhood of exactly one vertex of $C$. A perfect code is some time called an efficient dominating set. The existence or non-existence of perfect codes have been considered for many graphs. See the introduction of \cite{aabfk-2016} for some references.

The vertex set of the \emph{$n$-cube} $Q_n$ is the set $\mathbb{B}_n$ of  binary strings of length $n$, two vertices being adjacent if they differ in precisely one position. Classical 1-error correcting codes and perfect codes are codes and perfect codes in the graph $Q_n$.
The concatenation of strings $\bm{x}$ and $\bm{y}$ is noted $\bm{x}||\bm{y}$ or just $\bm{x}\bm{y}$ when there is no ambiguity. A string $\bm{f}$ is a \emph{substring} of a string $\bm{s}$ if there exist strings $\bm{x}$ and $\bm{y}$, may be empty, such that $\bm{s}=\bm{x}\bm{f}\bm{y}$.

A {\em Fibonacci string} of length $n$ is a binary string $\bm{b}=b_1\ldots b_n$ with $b_i\cdot b_{i+1}=0$ for $1\leq i<n$. In other words a Fibonacci string is a binary string without $11$ as substring.
The {\em Fibonacci cube} $\Gamma_n$ ($n\geq 1$) is the subgraph of $Q_n$ induced by the Fibonacci strings of length $n$.
Fibonacci cubes were introduced as a model for interconnection networks~\cite{hsu-93} and received a lot of attention afterwards. 
These graphs also found an application in theoretical chemistry. See the survey \cite{survey} for more results and applications about Fibonacci cubes.

The sets $\{00\}$ and $\{010,101\}$ are perfect codes in respectively $\Gamma_2$ and $\Gamma_3$. In a recent paper  \cite{aabfk-2016} Ashrafi and his co-authors proved the non-existence of perfect codes in  $\Gamma_n$ for $n\geq 4$. As an open problem the authors suggest to consider the existence of perfect codes in generalization of Fibonacci cubes.
The most complete generalization  proposed in \cite{ikr} is, for a given string $\bm{f}$, to consider $\Gamma_n(\bm{f})$ the subgraph of $Q_n$ induced by strings that do not contain $\bm{f}$ as substring. Since Fibonacci cubes are $\Gamma_n(11)$ the most immediate generalization \cite{hsuliu,Zag} is to consider  $\Gamma_n({1^s})$ for a given integer $s$. We will prove the existence of perfect codes in $\Gamma_n({1^s})$ for an infinite family of parameters $(n,s)$. 

It will be convenient to consider the binary strings of length $n$ as vectors   of $\mathbb{F}^n$ the vector space of dimension $n$ over the field $F=\mathbb{Z}_2$ thus to associate to a string $x_1 x_2 \dots x_n$ the vector $\theta(x_1 x_2 \dots x_n)=(x_1,x_2,\ldots,x_n)$.
The \emph{Hamming distance} between two vectors $\bm{x},\bm{y} \in \mathbb{F}^n$, $d(\bm{x},\bm{y})$ is the number of coordinates in which they differ.
The $\emph{parity function}$ is the function from $\mathbb{F}^n$ to $\mathbb{Z}_2$ defined by $\pi(\bm{x})=\pi(x_1,x_2,\ldots,x_n)= x_1+x_2+\ldots+x_n$.
By the correspondence $\theta$ we can define the sum $\bm{x}+\bm{y}$, the Hamming distance $d(\bm{x},\bm{y})$ and the parity $\pi(\bm{x})$ of  strings in $\mathbb{B}_n$. Note that Hamming distance is the usual graph distance in $Q_n$.
The complement of a string  $\bm{x}\in \mathbb{B}_n$ is the string $\overline{\bm{x}}=\bm{x}+1^n$.

We will first recall some basic results about perfect codes in $Q_n$.
Since $Q_n$ is a regular graph of degree $n$ the existence of a perfect code of cardinality $|C|$ implies $|C|(n+1)=2^n$ thus a necessary condition of existence is that $n+1$ is a power of 2 thus that $n=2^p-1$ for some integer $p$.

For any integer $p$ Hamming \cite{Ha1950} constructed, a linear subspace of $\mathbb{F}^{2^p-1}$ which is a perfect code. It is easy to prove that all linear perfect codes are Hamming codes.\\
In 1961  Vasilev \cite{{Va1962}}, and later many authors, see \cite{Co1,Sol2008} for a survey, constructed perfect codes which are not linear codes.
Let us recall Vasilev's construction of perfect codes.

\begin{theorem}\label{thvas}\cite{Va1962}
Let $C_r$ be a perfect code of $Q_r$. Let $f$ be a function from $C_r$ to $\mathbb{Z}_2$  and $\pi$ be the $\emph{parity function}$.
Then the set $C_{2r+1}= \left\{\bm{x}||\pi(\bm{x})+f(\bm{c})||\bm{x}+\bm{c};\bm{x}\in \mathbb{B}_r,\bm{c}\in C_r\right\}$ is a perfect code of $Q_{2r+1}$
 \end{theorem}
We recall  also the proof of Theorem \ref{thvas} in such a way our article will be self contained. 
\begin{proof}
Fist notice that $|C_{2r+1}|=2^r|C_r|=2^r\frac{2^{r}}{r+1}=\frac{2^{2r+1}}{2r+2}$. Thus if is sufficient to prove that the distance between to different elements of $C_{2r+1}$ is at least 3.\\
Consider $d(\bm{x}||\pi(\bm{x})+f(\bm{c})||\bm{x}+\bm{c}, \bm{x'}||\pi(\bm{x'})+f(\bm{c'})||\bm{x'}+\bm{c'})=d_1+d_2+d_3$  where $d_1=d(\bm{x},\bm{x'})$, $d_2=d(\pi(\bm{x})+f(\bm{c}),\pi(\bm{x'})+f(\bm{c'}))$ and $d_3=d(\bm{x}+\bm{c},\bm{x'}+\bm{c'})$.\\
If $d_1=0$ then $\bm{x}= \bm{x'}$ thus $d_3=d(\bm{c},\bm{c'})\geq 3$.\\
If $d_1=1$ and $\bm{c}= \bm{c'}$ then $d_2=d_3=1$\\
If $d_1=1$ and $\bm{c} \neq \bm{c'}$ then $d_3\geq 2$ otherwise $d(\bm{c},\bm{c'})\leq2$\\
If $d_1=2$ then $d_3\neq 0$ otherwise $d(\bm{c},\bm{c'})=2$\\
Thus $d=d_1+d_2+d_3\geq3$.
\end {proof}
\qed                                      
 
If $f(\bm{c})=0$ for any $\bm{c}\in C_r$ we obtain the classical inductive construction of Hamming codes with $C_1=\{0\}$ as basis.

In the next section we will use this construction starting from the Hamming code in $Q_r$ as  $C_r$ and a function $f$ chosen  in such way that the strings of the constructed code $C_{2r+1}$ has not a too big number of consecutive 1's.

\section{Main Result}
\begin{lemma}
Let $m$ be an integer. Let $A_0$ be the set of strings $A_0=\{0^{m+1}\bm{y};\bm{y}\in \mathbb{B}_{m}\}$. For $i\in \{1,\dots,m\}$ let $A_i=\{\bm{z}10^{m+1}\bm{y}; \bm{z}\in \mathbb{B}_{i-1}, \bm{y}\in \mathbb{B}_{m-i}\}$.
Then the sets $A_i$ are disjoint and any string of ${B}_{2m+1}$ containing $0^{m+1}$ as substring belongs to a $A_i$. 
\end{lemma}
\begin{proof}
Let $\bm{x}$ be a string of ${B}_{2m+1}$ containing $0^{m+1}$ as substring and $i$ be the minimum integer such that $x_{i+1}x_{i+2}\dots x_{i+m+1}=0^{m+1}$. Then $i=0$, and $\bm{x}$ belongs to $A_0$, or $m \geq i\geq 1$. In this case  $x_{i}=1$ thus $\bm{x}\in A_i$.
 Assume $\bm{x}\in A_i\cap A_j$ with $m\geq j>i\geq 0$ then $x_{j}=1$ thus $j\geq i+m+2>m$ a contradiction. 
\end{proof}
\begin{theorem}\label{thmain}
Let $n=2^p-1$ where $p\geq 2$ and let $s =3.2^{p-2}$. There exists a perfect code $C$ in $Q_n$ such that no elements of $C$ contains $1^s$ as substring.
 \end{theorem}
\begin{proof}
Let $m=2^{p-2}-1$ thus $2m+1=2^{p-1}-1$ and $s=3m+3$. Let $C_{2m+1}$ be a perfect code in $Q_{2m+1}$. Let $f$ be the function from  $\mathbb{B}_{2m+1}$ to $\mathbb{Z}_2$ defined by
\begin{itemize} 
\item $f(0^{m+1}\bm{y})=1$ for $\bm{y}\in \mathbb{B}_{m}$
\item $f(10^{m+1}\bm{y})=0$ for  $\bm{y}\in \mathbb{B}_{m-1}$
\item $f(\bm{z}10^{m+1}\bm{y})=\pi(\bm{z})$ for $\bm{z}\in \mathbb{B}_{i-1}$ and $\bm{y}\in \mathbb{B}_{m-i}$ for $i=2$ to $m$.
\item $f=0$ otherwise.
\end{itemize}
Note that from the previous lemma the function is well defined.
Let $C$ be the perfect code obtained from Vasilev's construction from $C_{2m+1}$ and $f$.
Assume there exists a string $\bm{d}$ in $C$ with $1^{3m+3}$ as substring. Therefore $\bm{d}$ is obtained from $\bm{x}=d_1d_2\dots d_{2m+1}$ and $\bm{c}\in C_{2m+1}$. Since $n= 4m+3$ note first that $d_{m+1}d_{m+2}\dots d_{3m+3}=1^{2m+2}$.
Let $i$ be the minimum integer such that $d_{i}d_{i+1}\dots d_{3m+i+2}=1^{3m+3}$. We consider 3 cases
\begin{itemize}
\item $i=1$ then $x=d_1d_2\dots d_{2m+1}=1^{2m+1}$ and $d_{2m+2}d_{2m+3}\dots d_{3m+3}=1^{m+2}$. Since  $\bm{c}+\bm{x}=1^{m+1}d_{3m+4}d_{3m+5}\dots d_{4m+3}$ we have $\bm{c}=0^{m+1}\bm{y} $ for some $\bm{y}\in \mathbb{B}_{m}$. Thus $f(\bm{c})=1$ and since  $\pi(\bm{x})=1$ we obtain  $d_{2m+2}= f(\bm{c})+\pi(\bm{x})=0$ a contradiction.

\item $i=2$ then $\bm{x}=01^{2m}$ and $d_{2m+2}d_{2m+3}\dots d_{3m+4}=1^{m+3}$. Since  $\bm{c}+\bm{x}=1^{m+2}d_{3m+5}d_{3m+6}\dots d_{4m+3}$ we have $\bm{c}=10^{m+1}\bm{y} $ for some $\bm{y}\in \mathbb{B}_{m-1}$. Thus $f(\bm{c})=0$ and since  $\pi(\bm{x})=0$ we obtain  $d_{2m+2}=f(\bm{c})+\pi(\bm{x})=0$ a contradiction.

\item $i\geq3$ then $\bm{x}=\bm{z}01^{2m-i+2}$ for $\bm{z}\in \mathbb{B}_{i-2}$ and $d_{2m+2}d_{2m+3}\dots d_{3m+2+i}=1^{m+i+1}$. Since $\bm{c}+\bm{x}=1^{m+i}d_{3m+i+3}d_{3m+i+4}\dots d_{4m+3}$ we have $\bm{c}=\overline{\bm{z}}10^{m+1}\bm{y} $ for some $\bm{y}\in \mathbb{B}_{m-i+1}$. Thus $f(\bm{c})=\pi(\overline{\bm{z}})$. Since  $\pi(\bm{x})=\pi(\bm{{z}})+\pi({1^{2m-i+2}})$ and $\pi(\overline{\bm{z}})+\pi(\bm{z})=\pi({1^{i-2}})$ we obtain  $d_{2m+2}= f(\bm{c})+\pi(\bm{x})=\pi({1^{2m}})=0$ a contradiction.
\end{itemize}
Therefore there exists no string $d$ in $C$ with $1^{3m+3}$ as substring.
\end{proof}
\qed  
 \begin{corollary}\label{cormain}
Let $n=2^p-1$ where $p\geq 2$ and let $s \geq 3.2^{p-2}$. There exists a perfect code in $\Gamma_n(1^s)$.
\end{corollary}       
\begin{proof}
Indeed let $C$ be a perfect code in $Q_n$ such that no element of $C$ contains $1^{3.2^{p-2}}$ as substring. 
The strings of $C$ are in $V(\Gamma_n(1^s))$. Let $x$ be a vertex of $V(\Gamma_n(1^s))$. If $x\notin C$ then $x$ is adjacent in $Q_n$ to a vertex $c$ in $C$. Note that $x$ and $c$ are also adjacent in $\Gamma_n(1^s)$ thus $C$ is a dominating set of $\Gamma_n(1^s)$. If $c$ and $c'$ are two strings of $C$ then $d_{\Gamma_n(1^s)}(c,c')\geq d_{Q_n}(c,c')\geq 3$. Therefore $C$ is a perfect code in  $\Gamma_n(1^s)$.
\end{proof}
\section{Concluding remark and open problems}
Whenever $n=2^p-1$ it will be interesting to determine  
the minimum $s$ such that there exists a perfect code in $\Gamma_n(1^s)$.

Corollary \ref{cormain} is not always the best result possible. For example for $n=7$ the code $C_7$ obtained in  Vasilev's construction starting from $C_3=\{000,111\}$ with $f(000)=f(111)=1$ is a perfect code in $\Gamma_n(1^5)$. Indeed
\begin{itemize}
\item $11111ab$ or $0011111$ cannot be in $C_7$ since the  $P(111)+1=P(001)+1=0$
\item $011111a$ cannot be in $C_7$ since the possible codewords begining with $011$ are $0111011$ and $0111100$.
\end{itemize}
Note that that all strings of this code are obtained from strings in the Hamming code of length 7 by a translation of $0001000$. This simple idea can be generalized but is less efficient than our result in the general case.

We propose also the following conjecture:
\begin{conjecture}
For $n\geq 3$ and $s\geq 1$ if  $C$ is a perfect code in $\Gamma_n(1^s)$ then $n=2^p-1$ for some integer $p$ and furthermore $C$ is a perfect code in $Q_n$.
\end{conjecture}.

\end{document}